\documentclass[10pt]{article}

\usepackage{amsfonts,amssymb,amsmath}
\usepackage{dsfont}
\input{xy}
\usepackage[all]{xy}
\xyoption{all}
\xyoption{poly}

\newtheorem{thm}{Theorem}
\newtheorem{cor}[thm]{Corollary}
\newtheorem{lem}[thm]{Lemma}
\newtheorem{pro}[thm]{Proposition}
\newtheorem{defi}[thm]{Definition}
\newtheorem{rem}[thm]{Remark}

\newenvironment{proof}{\noindent \textbf{{Proof.}} \sf}

\def\qed{\hfill $\diamond$ \bigskip}

\def\B{{\mathcal B}}
\def\C{{\mathcal C}}
\def\D{{\mathcal D}}

\def\F{{\mathcal F}}

\def\K{{\mathcal K}}

\def\U{{\mathcal U}}

\def\lim{\mathop{\rm lim}\nolimits}

\def\H{\mbox{{\rm H}}}

\def\Im{\mathsf{Im}}

\def\aut{\mathsf{Aut}}

\def\gal{\mathsf{Gal}}

\def\deg{\mathsf{deg}}
\def\pigr{\Pi^{\mbox{\sf \footnotesize  gr}}_1(\B,b_0)}

\begin{document}

\sf

\title{On universal gradings,  versal gradings and Schurian generated  categories}
\author{Claude Cibils, Mar\'\i a Julia Redondo and Andrea Solotar
\thanks{\footnotesize This work has been supported by the projects  UBACYTX475, PIP-CONICET 112- 200801-00487, PICT-2011-1510 and MATHAMSUD-NOCOSETA.
The second and third authors are  research members of
CONICET (Argentina).}}

\date{}

\maketitle

\begin{abstract}
Categories over a field $k$ can be graded by different groups in a connected way; we consider morphisms between these gradings in order to define the fundamental grading group. We prove that this group is isomorphic to the fundamental group {\it\`{a} la Grothendieck} as considered in previous papers. In case the $k$-category is Schurian generated we prove that a universal grading exists. Examples of non Schurian generated categories with universal grading,  versal grading or none of them are considered.

\end{abstract}


\section{\sf Introduction}

We will consider throughout this paper small categories $\B$ enriched over $k$-vector spaces, namely small categories whose sets of  morphisms are $k$-vector spaces and such that  composition is bilinear; they are called $k$-categories. In particular the endomorphisms of each object are $k$-algebras and the spaces of morphisms are bimodules.

A grading $X$ of $\B$ is a direct sum decomposition of each vector space of morphisms into homogeneous components which are indexed by elements of the structural group $\Gamma(X)$, such that the composition of $\B$ is compatible with the product of the group.  By definition the grading is connected if any element of the structural group is the degree of a homogeneous closed walk.

 Observe that a homotopy theory of loops taking into account the $k$-linear structure of a $k$-category is not available. Previous papers -- see \cite{CRS ANT 10,CRS ART 12,CRS DOC 11} -- show that gradings may be considered as a substitute, see also Remark \ref{theory}; we have inferred an intrinsic  fundamental group of a $k$-category which takes into account all its  connected gradings and which is functorial with respect to full subcategories -- see \cite{CRS PAMS 12}.

 In this paper we introduce morphisms between gradings, which are group maps  $\mu$ between the structural groups  such that there exists a homogeneous automorphism functor of $\B$ which is the identity on objects and which induces $\mu$ on the degrees of the homogeneous closed walks, see Definition \ref{morphism of grading}.

This leads to the definition of the fundamental grading group. In case there exists a universal grading $U$ we obtain that $\Gamma(U)$ is isomorphic to this group. If a versal grading $V$ exists (see Definition \ref{universal versal}) then the intrinsic fundamental group is a subgroup of the fixed points group with respect to the automorphisms of $V$.

In Section \ref{fundamental grading} we study gradings as mentioned above while in Section \ref{fundamental intrinsic} we recall the intrinsic fundamental group \emph{\`{a} la Grothendieck} obtained through Galois coverings and smash products -- see  \cite{CRS ANT 10,CRS ART 12}. In Section \ref{comparison} we prove that the fundamental grading group considered in Section \ref{fundamental grading} is indeed isomorphic to the intrinsic one.

In other words the gradings as considered in this paper provide a different approach to the theory of the intrinsic fundamental group developed in \cite{CRS ANT 10,CRS ART 12,CRS DOC 11}. Using the introduced morphisms of gradings, we show in Theorem \ref{universal fundamental} that in case a universal grading exists its structural group is isomorphic to the fundamental group.

This result is important by itself as well as with respect to the following: the proof of the analogous fact concerning universal coverings in \cite[Proposition 4.3]{CRS ART 12} is actually incomplete. Indeed, towards the end of the proof a set of elements is obtained which should provide an automorphism. For this purpose those elements need to constitute a coherent family  which means that they should correspond with  themselves through canonical group maps. However this is not proved in the cited paper and the fact is not clear.

Note that in \cite{gr} a relation between gradings and coverings is established for quivers with relations, see also \cite{grma}. In this paper we consider the intrinsic context, where the categories are not given by a presentation.

In the rest of the paper we apply the theory developed before to Schurian generated categories and other examples.
By definition, a non-zero morphism fron one objetc $b$ to an object $b'$ is called Schurian if the space of morphisms from $b$ to $b'$ is one-dimensional. The intersection of all the $k$-subcategories containing them is called the Schurian generated subcategory of $\B$.  When this intersection is $\B$ the category is called Schurian generated. Note that in the framework of $k$-categories presented by a quiver with admissible relations  this corresponds to constricted  algebras (also called constrained algebras) -- see \cite{bama}. We prove that a Schurian generated category admits a universal grading. To this end we show that its fundamental grading group is the structural group of a grading which is not necessarily connected. This grading can be restricted to the image of its degree map, providing this way a connected grading which is universal.

Finally we consider four examples. The first one is at the source of the theory of the fundamental group \emph{\`{a} la Grothendieck}. Indeed, Bongartz, Gabriel \cite{boga} and Martinez-Villa, de la Pe\~{n}a \cite{MP} have considered a fundamental group attached to a presentation by a quiver with admissible relations of a category. It is not an invariant of the isomorphism class, it depends on the chosen presentation -- see for instance \cite{buca}. We provide a one-parameter family of $k$-categories  $\B_q$ that  has a universal cover. This one-parameter family of deformations is actually trivial and the fundamental group is infinite cyclic. The second example  is the Kronecker category, which has no universal grading, but admits a versal grading. Its fundamental group is trivial. Then we show that for a monomial Schurian category the above mentioned fundamental group attached to a presentation  is isomorphic to the intrinsic fundamental group. Finally we consider the group algebra in characteristic $p$ of the cyclic group $C_p$ of order $p$, which has neither universal nor versal grading; its fundamental group is the product of the infinite cyclic group and  $C_p$.

\section{Fundamental grading group}\label{fundamental grading}

In this section we first recall definitions from \cite{CRS ANT 10, CRS ART 12}. Then we introduce morphisms of gradings, universal and versal gradings.  This leads to the definition of the fundamental grading group.

\subsection{Homogeneous walks and grading homotopy}\label{homogeneous walks}\sf

For a $k$-category $\B$ the set of objects is denoted $\B_0$ while  ${}_{b'}\!\B_b$ is the vector space of morphisms having source $b$ and target ${b'}\!$. For $f\in {}_{b'}\!\B_b$  we write $\sigma(f)=b$ and $\tau(f)={b'}\!$. A $k$-category is  \textbf{connected} if the graph of its non-zero morphisms is connected.

\begin{defi} \label{gradings}\sf   A \textbf{grading} $X$ of a $k$-category $\B$ by a group $\Gamma(X)$ is a direct sum decomposition of each morphism space $${}_{b'}\!\B_b = \bigoplus_{s\in \Gamma(X)}X^s{}_{b'}\! \B_b$$ such that $$\left(X^t{}_{b''}\B_{b'}\!\right)\left(X^s{}_{b'}\! \B_b\right) \subset X^{ts}{}_{b''}\B_b.$$
The group $\Gamma(X)$ is called the structural group of the grading. A \textbf{homogeneous} morphism $f$ of \textbf{degree $s$} is a non-zero morphism lying in a \textbf{homogeneous component} $X^s{}_{b'}\! \B_b$. We write $\deg_X f=s$.
\end{defi}

A \textbf{virtual morphism} is a pair  $(f,\epsilon)$ where $f$  is a non-zero morphism and $\epsilon$ is $1$ or $-1$. We set  that the source $\sigma$ and the target $\tau$ of $(f,-1)$ are reversed with respect to those of $f$ while for $(f,1)$ they remain unchanged. We do not compose virtual morphisms.

A \textbf{walk} $w$ with source $b$ and target ${b'}\!$ is a sequence of \textbf{concatenated} virtual morphisms
$$w=(f_n,\epsilon_n),\dots,  (f_2,\epsilon_2),(f_1,\epsilon_1)$$
verifying $$\sigma(f_1,\epsilon_1)=b, \tau(f_1,\epsilon_1)=\sigma(f_2, \epsilon_2),\dots, \tau(f_n,\epsilon_n)={b'}\!.$$
The \textbf{formal inverse} of $w$ is $w^{-1}=(f_1,-\epsilon_1),(f_2,-\epsilon_2),\dots,  (f_n,-\epsilon_n)$ with source ${b'}\!$ and target $b$.

We say that two walks $w, w'$ are  \textbf{concatenable} if the target of $w$ coincides with the source of $w'$.  We denote $w'w$ their concatenation.

A \textbf{virtual homogeneous morphism} with respect to a grading is a virtual morphism $(f,\epsilon)$ where $f$ is homogeneous. A \textbf{homogeneous walk} is a walk made of virtual homogeneous morphisms. The set of homogeneous walks from $b$ to $b'$ with respect to a grading $X$ is denoted ${}_{b'}\!HW(\B,X)_b$.

For a grading $X$ the \textbf{degree of a homogeneous walk} is by definition $$\deg_X w =\left( \deg_X  f_n \right)^{\epsilon_n}\dots \left(\deg _X f_2 \right)^{\epsilon_2}\left(\deg_X  f_1 \right)^{\epsilon_1}.$$
\begin{rem}
\label{image is a group}
Let $w$ and $w'$ be concatenable walks and let $b\in\B_0$. The following facts are easy to verify.
\begin{itemize}
\item $\deg_X w^{-1} = \left(\deg_X w\right)^{-1}.$
\item
$\deg_X w w'= \deg_X w \ \deg_X w'.$
\item
The image of the inferred degree map
$\deg_X: {}_{{b}}HW_b(\B,X)\to \Gamma(X)$ is a subgroup of $\Gamma(X)$.
\end{itemize}
\end{rem}

\begin{defi}
The grading is \textbf{connected} if $\deg_X: {}_{{b'}}HW_b(\B,X)\to \Gamma(X)$ is surjective for any objects $b$ and ${b'}$.
\end{defi}

Note that  if $\B$ is connected the degree map is surjective for a given pair of objects if and only if the degree map is surjective for any pair of objects.

\begin{rem}\label{theory}
 For a $k$-category there is no definition of homotopy of loops available as in algebraic topology taking into account its  linear structure. A first attempt  -- see \cite{boga,MP} --  is to choose  a presentation of the $k$-category and to define homotopy using minimal relations of the presentation.

 However this way of doing provides a homotopy group which varies with the chosen presentation -- see \cite{buca}. Due to this fact we have previously introduced an intrinsic fundamental group defined using coverings which are issued from connected gradings. Indeed, gradings and coverings are closely related as we will recall in the following.

In this context observe that a given grading $X$ provides a partial substitute for a homotopy theory of loops. Indeed we can consider  among homogeneous walks with same source and target an "$X$-homotopy" relation given by the degree map, namely two homogeneous walks are  \textbf{$X$-homotopic} if they have the same $X$-degree. In case $X$ is connected the $X$-homotopy classes form a group which is isomorphic to the structural group $\Gamma(X)$.
\end{rem}

\subsection{Universal and versal gradings}
Let $X$ and $Y $ be gradings of a $k$-category $\B$. A  $k$-automorphism functor $J$ of $\B$ is called \textbf{homogeneous} from $X$ to $Y$ if $J$  is the identity on objects and if the image of $X$-homogeneous morphisms of $\B$ are $Y$-homogeneous.

\begin{defi}
\label{J extended}
Let  $X, Y$ be two gradings of $\B$, $b, b' \in \B_0$.  A homogeneous automorphism $J$ from $X$ to $Y$ induces a  map   $$HW(J): {}_{b'}\!HW(\B,X)_b\to {}_{b'}\!HW(\B,Y)_b$$
$$HW(J)(w)=(J(f_n),\epsilon_n),\dots, (J(f_1),\epsilon_1)$$ for $w=(f_n,\epsilon_n),\dots, (f_1,\epsilon_1)$ an $X$-homogeneous walk.
\end{defi}
However in general a map from $\Gamma(X)$ to $\Gamma(Y)$ induced by $J$ does not exist since the image by $J$ of two homogeneous closed walks of same $X$-degree may  have different $Y$-degrees.

\begin{defi}\label{morphism of grading}
 Let $b_0$ be a base object of a connected $k$-category $\B$ and let $X$ and $Y$ be connected gradings of $\B$. A \textbf{morphism} $\mu : X\to Y$ is a group map $\mu :\Gamma(X)\to \Gamma(Y)$ such that there exists a (non-necessarily unique)  homogeneous automorphism functor $J$ from $X$ to $Y$ making commutative the following diagram

$$
\xymatrix{
{}_{b_0}HW(\B,X)_{b_0} \ar@{->>}[d]_{\deg_X} \ar[r]^{HW(J)} & {}_{b_0}HW(\B,Y)_{b_0}\ar@{->>}[d]^{\deg_Y}\\
\Gamma(X) \ar[r]_{\mu}   &\Gamma(Y).}
$$

\end{defi}

\begin{defi}\label{universal versal}
A connected grading $U$ is \textbf{universal} if for any connected grading $X$ there exists a unique morphism  $\mu: U\to X$. A connected grading $V$ is \textbf{versal} if for any connected grading $X$ there exists at least one morphism $\mu: V\to X$.
\end{defi}

Concerning the notion of \emph{versal} see for instance the appendix by J.P. Serre in \cite{dure}.

Of course if a universal grading exists, it is unique up to an isomorphism of gradings. In general universal gradings do not exist as we shall see in the last section. Nevertheless we will prove  that Schurian generated $k$-categories admit universal gradings.

\subsection{Graded coherent families}

Next we will define the fundamental grading group; in the following section we will prove that it is isomorphic to the intrinsic fundamental group defined as the automorphism group of a fibre functor. In this subsection we will show  that the structural group of a universal grading is isomorphic to the fundamental grading group.
\begin{defi}\label{fgg}
Let $\B$ be a connected $k$-category, and let $b_0$ be a fixed object.
An element of the \textbf{fundamental grading group} $\Pi^{\mbox{\sf \footnotesize  gr}}_1(\B,b_0)$ is a family  $\gamma = \left(\gamma_X\right)_X$ where $X$ varies among the connected gradings,  $\gamma_X \in\Gamma(X)$, and which is  \textbf{graded coherent} namely $\mu(\gamma_X)=\gamma_Y$ for each morphism $\mu:X\to Y.$ The product is pointwise.
\end{defi}

\begin{thm}\label{universal fundamental}
Let $\B$ be a connected $k$-category with a given object $b_0$ admitting a universal grading $U$. Then
$$\pigr\simeq \Gamma(U).$$
\end{thm}
\begin{proof}
For each connected grading $X$  let $\mu_X : U \to X$ be the unique morphism of gradings. For any $\delta\in\Gamma(U)$ we associate the family $\left(\mu_X(\delta)\right)_X$ where $X$ varies among the connected gradings. This family is clearly graded coherent and we obtain this way a group homomorphism $\Gamma(U)\to\pigr$. Its inverse associates to a family $\left(\gamma_X\right)_X$  the element $\gamma_U$. \qed
\end{proof}

For a grading $X$ of $\B$ we denote ${\mathsf{Fix}}(X)$ the subgroup of  $\Gamma(X)$ of elements $\sigma_X$ such that $\mu(\sigma_X)=\sigma_X$ for every  morphism of gradings $\mu:X\to X$.
\begin{pro}\label{fix}
Let $\B$ be a connected $k$-category with a given object $b_0$ and which admits a versal grading $V$. There is an injective group map
$$\pigr\to\mathsf{Fix} (V).$$
\end{pro}
\begin{proof}
To an element $\gamma=\left(\gamma_X\right)_X$ in $\pigr$ we associate $\gamma_V\in\Gamma(V)$ which lies in $\mathsf{Fix} (V)$ since $\gamma$ is a coherent graded family. In case $\gamma_V=1$ consider for each connected grading $X$ a map of gradings $\mu:V\to X$ which exists since $V$ is versal. Since the family is graded coherent we infer that $\gamma_X=\mu(\gamma_V)= \mu(1)=1$.\qed
\end{proof}

\section{Intrinsic fundamental group}\label{fundamental intrinsic}

In order to prove that the fundamental grading group considered above is isomorphic to the intrinsic fundamental group that we have considered in  \cite{CRS ANT 10,CRS ART 12} we provide for the convenience of the reader a brief account of the main tools and results concerning this theory;  complete proofs and details can be found in the cited papers.

\subsection{Galois coverings}
\label{coverings}\sf The $b$-\textbf{star} of an object $b$ in a $k$-category $\B$ is the direct sum of the morphism spaces having  $b$ as a source or as  a target. Let $\C$ be a connected and non-empty $k$-category. A $k$-functor $F:\C\to\B$ is a \textbf{covering} if it is surjective on objects and if it induces $k$-isomorphisms between corresponding stars; more precisely  each  $b$-star in $\B$ is isomorphic - using  $F$ - to the $c$-star in $\C$  for any  $c\in\C_0$ such that $F(c)=b$. As a consequence a covering is always faithful. Note also that if a covering of $\B$ exists then $\B$ is connected.

\label{morphisms}\sf Let $F:\C\to\B$ and $G:\D\to\B$ be coverings of a $k$-category $\B$. A \textbf{morphism} from $F$ to $G$ is a pair of $k$-functors $(H,J)$ where $H:\C\to\D$ and $J:\B\to\B$ are such that :\\
- $J$ is an automorphism of $\B$ which is the identity on objects,\\
- $GH=JF$.

\begin{rem}
Morphisms of the form $(H,J)$ are called $J$-morphisms. Due to an observation of P. Le Meur -- see \cite{le,le1,le2} -- $\mathds{1}$-morphisms are not enough in order to insure that some coverings are isomorphic: indeed if $F$ is a covering and $J$ is an automorphism of $\B$ which is the identity on objects, $JF$ is a covering which is isomorphic to $F$ but in general not through a $\mathds{1}$-morphism.
\end{rem}

For a covering $F:\C\to\B$ let $\aut_1 F$ be the group of automorphisms of $\C$ such that $HF=F$. It acts on the $F$-fibre of each $b\in\B_0$, namely it acts on $F^{-1}(b)=\{c \in \C_0 \mid\ F(c)=b\}$. This action is free due to a result by P. Le Meur stating that two $J$-morphisms which coincide on some object are equal -- see \cite[2.9]{CRS ART 12}.

\begin{defi}
\label{Galois}
A covering is  \textbf{Galois} if the action of $\aut_1 F$ is transitive on some fibre, or equivalently, if the action is transitive on any fibre.
\end{defi}

A connected grading $X$ of $\B$ provides a Galois covering through a $k$-category called \textbf{smash product} $\B\#X$ defined as follows. The set of objects is the cartesian product $\B_0\times \Gamma(X)$ while the vector space of morphisms from $(b,s)$ to $(b',t)$ is $X^{t^{-1}s } {}_{b'}\B_b$ -- see \cite{CM, CRS ANT 10, CRS ART 12}. Actually it can be proven that any Galois covering $F$ is isomorphic to a smash product through a $\mathds{1}$-morphism which is not canonical, it depends on the choice of an object in each fibre of $F$. We consider $\gal^{\#}(\B,b_0)$ the full subcategory of the category of Galois coverings $\gal(\B,b_0)$ whose objects are coverings of type $F_X : \B\# X\to\B$.

Given a connected grading $X$ and the corresponding Galois covering $F_X$,  the structural group $\Gamma(X)$ is identified to $\aut_1 (F_X)$ as follows: the action on objects is given by left multiplication on the second component, and on morphisms is provided by the identical translation of the corresponding homogeneous components. A map of smash coverings $(H,J): F_X\to F_Y$ is given on objects by $H(b,s)=(b,H_b(s))$ and by $J$ on morphisms.

Following methods closely related to the way in which the fundamental group is considered in algebraic geometry after A. Grothendieck and C. Chevalley -- see for instance the book by R. Douady and A. Douady  [9] -- we have defined in \cite{CRS ANT 10, CRS ART 12} the following:
 \begin{defi}\label{fundamental}
Let $b_0$ be an object of a connected $k$-category $\B$ and consider the fibre functor $\Phi^{\#}$ which assigns to a smash product Galois covering $F_X$ its fibre $F_X^{-1}(b_0)$ in the category of sets. The \textbf{intrinsic fundamental group} $\Pi_1(\B,b_0)$ is $\aut \ \Phi^\#.$
\end{defi}

\subsection{Morphisms of smash products}
\label{lambda}\sf
 In order to have a concrete description of the elements of the previous intrinsic fundamental group, we first recall that we associate a unique group map  $$\lambda_{(H,J)}: \aut_1 F\to\aut_1 G$$ to a morphism of Galois coverings $(H,J) : F\to G$  such that $Hf=\lambda_{(H,J)}(f)H$ for every $f\in\aut_1F$.
Note the following facts:
\begin{itemize}
\item  $\lambda$ is functorial with respect to composition of morphisms of coverings, namely
$$\lambda_{(H'H,J'J)}= \lambda_{(H',J')}\lambda_{(H,J)}$$
\item
$\lambda_{(qH,J)}= q\lambda_{(H,J)} q^{-1}$ for $q\in\aut_1 F$.
\end{itemize}

Let $J$ be an automorphism of $\B$ which is the identity on objects and let $X$ and $Y$ be connected gradings of $\B$. In case there exists a $J$-morphism from $F_X$ to $F_Y$ we consider the \textbf{normalized} one verifying $N(b_0,1)=(b_0,1)$. Observe  that for any $(H,J)$ we have $N=H_{b_0}(1)^{-1} H$. We set  $\mu_J=\lambda_{(N,J)}$. According to the above formula for $\lambda_{(qH,J)}$ we note the following for any morphism $(H,J)$:

\begin{equation}\label{mu lambda}\tag{1}\mu_J=H_{b_0}(1)^{-1}\lambda_{(H,J)}H_{b_0}(1).
\end{equation}

\label{morphisms of smash}

More precisely
    $H_b(s)= \lambda_{(H,J)}(s)H_b(1)$
since
    $$ (b, H_b(s))=H(b,s)=Hs(b,1)=\lambda_{(H,J)}(s)H(b,1)=$$ $$\lambda_{(H,J)}(s)(b,H_b(1))= (b, \lambda_{(H,J)}(s)H_b(1)).$$

\label{coherent}
The following result is proven in detail in \cite[2.10]{CRS PAMS 12}. The fundamental group $\Pi_1(\B,b_0)$ is isomorphic to the group of  \textbf{coherent families}  $\sigma=(\sigma_X)_X$ where $X$ varies over all the connected gradings, $\sigma_X\in\Gamma(X)$ and $\mu_J(\sigma_X)=\sigma_Y$ for any $J$ such that there exists a $J$-morphism from $F_X$ to $F_Y$.

\subsection{Universal coverings}
\label{universal}\sf
 We consider pointed coverings of $(\B,b_0)$, namely coverings $F: (\C,c)\to(\B,b_0)$ where $c$ is an object such that $F(c)=b_0$.  In the previously cited papers a pointed Galois covering $U: (\U,u)\to(B,b_0)$ is called \textbf{universal} if for any Galois pointed covering $F:(\C,c)\to(\B,b_0)$ there exists a unique morphism $(H,\mathds{1})$ from $U$ to $F$ such that $U(u)=c$.

\begin{rem}
 When defining a universal  covering we require the existence of a unique pointed morphism $(H,\mathds{1})$. Instead, we could have required the existence of a pointed morphism $(H,J)$  without specifying the value of $J$. Doing so uniqueness is impossible to fulfill  since once a morphism exists any multiple of it by a non-zero scalar is still a morphism.
 \end{rem}

Universal coverings do not always exist -- see for instance \cite{CRS ART 12}. Nevertheless Schurian categories admit universal coverings -- see \cite{bu1,CRS DOC 11}.

As mentioned in the introduction, the proof in \cite[Proposition 4.3]{CRS ART 12} of the following result is incomplete:  for a $k$-category admitting a universal covering, its intrinsic fundamental group  should be isomorphic to the automorphism group of its universal covering. At the end of the proof of this result we obtain a family of elements which should provide an automorphism. This is correct only if the obtained family is coherent, which means that the elements correspond each other through the canonical group maps;  however this fact is not clear.

Instead it appears more natural to consider universal gradings. Indeed, Theorem \ref{universal fundamental} states that in case a universal grading exists its structural group is isomorphic to the fundamental grading group. In turn the latter is isomorphic to the intrinsic fundamental group according to the next section.

\section{Both fundamental groups are isomorphic}\label{comparison}

We will prove next that the two fundamental groups considered previously are isomorphic. We note that both of them consist of families of elements of the structural grading groups, where the families are respectively graded coherent and coherent. In Proposition \ref{same mu}  we will prove that the set of morphisms involved in the graded coherent and coherent requirements coincide. As an immediate consequence we infer  that the fundamental groups are isomorphic.

\begin{lem}\label{J is homogeneous}
Let $\B$ be a $k$-category, let $X$ and $Y$ be two connected  gradings and let $(H,J) :F_X\to F_Y$ be a morphism of the Galois coverings obtained through the smash product.
The automorphism $J$ is homogeneous from $X$ to $Y$  for any $X$-homogeneous morphism $f$ and we have that
$$\deg_Y J(f)=  H_{b'}(1)^{-1}\ \lambda_{(H,J)}(\deg_Xf)\ H_b(1).$$
\end{lem}
\begin{proof}
Let $f\in X^d {}_{b'}\B_b$. By definition of the smash product $f$ provides a morphism (still denoted $f$) from $(b,1)$ to $(b', d^{-1})$ lying  in the $F_X$-fibre of $f$.

Note that only homogeneous morphisms are lifted as a unique morphism, otherwise the pre-image of $f$ once a source object is fixed is a sum of morphisms according to the homogeneous components of $f$.

Since $J$ is the identity on objects,  $H(f)$ is a morphism in $\B\#Y$ from $(b, H_b(1))$ to $(b',H_{b'}(d^{-1}))= \left(b',\lambda_{(H,J)}(d^{-1})H_{b'}(1)\right)$. The morphism $F_Y(H(f))$ is $Y$-homogeneous from $b$ to $b'$ of degree \[\left(H_{b'}(1)\right)^{-1}\ \lambda_{(H,J)}(d)\ H_b(1).\]\qed

\end{proof}
\begin{pro}\label{deg mu}
Let $(H,J)$ be a morphism of Galois coverings from $F_X$ to $F_Y$ where $X$ and $Y$ are connected gradings of a  $k$-category $\B$. The group morphism $\mu_J: \Gamma(X)\to\Gamma(Y) $  is a morphism of gradings from $X$ to $Y$.
\end{pro}
\begin{proof}
Let $w$ be a homogeneous closed walk at $b_0$. We are going to prove that
$$\deg_YJ(w)= \mu_J \left(\deg_X w\right).$$
If $f$ is an $X$-homogeneous morphism then
$$\deg_Y J(f)=  H_{b'}(1)^{-1}\ \lambda_{(H,J)}(\deg_Xf)\ H_b(1).$$
We assert that this formula also holds for a homogeneous walk
$$w= \left(f_{n+1},\epsilon_{n+1}\right), \dots ,\left(f_1,\epsilon_1\right)$$
from $b_1$ to $b_{n+2}$. Indeed, the formula is verified  by induction by inspecting the four cases according to the values of $\epsilon_n$ and $\epsilon_{n+1}$.

In particular if the homogeneous walk is closed at $b_0$ we obtain the required formula using the relation between $\mu_J$ and $\lambda_{(H,J)}$ according to  \eqref{mu lambda} in \ref{lambda}\qed

\end{proof}

\begin{pro}\label{same mu}
Let $\B$ be a connected $k$-category with a base object $b_0$ and let $X$ and $Y$ be two connected  gradings. For any morphism of gradings $\mu:X\to Y$  there exists a morphism of Galois coverings $(H,J) :F_X\to F_Y$ such that $\mu_{J} = \mu$.

\end{pro}

\begin{proof}
Since $X$ is connected we can choose a family $(v_b)_{b\in\B_0}$ where $v_b\in {}_{b}HW(\B, X)_{b_0}$ and $\deg_X v_b=1$.
For an $X$-homogeneous morphism $f$ let $${}^v\!f ={v}_{\tau(f)}^{-1} f {v}_{\sigma(f)}.$$
The definition of a morphism of gradings $\mu:X\to Y$
includes the existence of a homogeneous automorphism $J$ from $X$ to $Y$ such that for any closed homogeneous walk $w$,
\[ \deg_YJ(w)= \mu\left(\deg_Xw \right). \]
In particular
$$\deg_Y J({}^v\!f)=\mu\left(\deg_X({}^v\!f)\right)=  \mu\left( \deg_X  {v}_{\tau(f)}^{-1} \ \deg_Xf\  \deg_X {v}_{\sigma(f)}\right)
= \mu\left(\deg_X f \right).$$
On the other hand
$$\deg_Y J({}^v\!f)= \deg_Y\  J\left({v}_{\tau(f)}^{-1} \ f \ {v}_{\sigma(f)}\right)=
\deg_Y J\left({v}_{\tau(f)}^{-1}\right)  \deg_Y J(f) \deg_Y J\left({v}_{\sigma(f)}\right).$$
We set $h_b=\left(\deg_Y J\left({v}_{b}\right)\right)^{-1}.$ Then
\begin{equation}
\label{degxy}\tag{2}
\mu\left(\deg_X  f  \right)= h_{\tau(f)} \    \deg_Y J (f) \  h_{\sigma(f)}^{-1}.
\end{equation}
We define $H:\B\#X \to \B\#Y$ on objects by $H(b,s)=(b,\mu(s)h_b)$. In order to have a $J$-morphism of smash products let $f\in{}_{(b',t)}\B\# X_{(b,s)}$, that is a morphism  $f\in X^{t^{-1}s} {}_{b'}\B_b$. We put $H(f)=J(f)$. We only have to check that $H(f)$  can be viewed as a morphism from $(b,\mu(s)h_b)$ to $(b',\mu(t)h_b')$. This is the case if and only if
$$  \deg_Y J(f) = [\mu(t)h_b']^{-1}\mu(s)h_b$$
and this equality is verified using \eqref{degxy} as follows:
$$   \deg_Y J(f) =  h_{b'} ^{-1}\ \mu\left(\deg_X(f)\right) \  h_{b} =  h_{b'}^{-1}\  \mu\left(t^{-1}s\right)\  h_{b}.$$
This morphism of smash products gives rise to a morphism of groups $\mu_J$.
According to the previous proposition $\mu_J:X\to Y$ is in turn also a morphism of gradings associated to $J$. Then $$\mu_J \left(\deg_X w\right)=\mu \left(\deg_X w\right).$$
Since $\deg_X$ is surjective we infer $\mu=\mu_J$.
\qed
\end{proof}

\begin{rem}\label{K}

Different morphisms of coverings can provide the same morphism of gradings as the following examples show.
Let $\K$ be the Kronecker category, that is the $k$-category with two objects $x$ and $y$ whose endomorphism algebras are $k$, with
no non-zero morphisms from $y$ to $x$, while ${}_y\K_x$ is two-dimensional. Note that once a basis $\{\alpha,\beta\}$ of  ${}_y\K_x$ is chosen the category   can also be described as the linear envelope of the path category of the quiver
\def\dar[#1]{\ar@<2pt>^\alpha[#1]\ar@<-2pt>_\beta[#1]}
  \entrymodifiers={!!<0pt,0.7ex>+}
$$\raisebox{.7ex}{\xymatrix{ {}_x\cdot\dar[r] & \cdot_y }}$$
Let $X$ be the grading with infinite cyclic structural group $\Gamma(X)$ generated by $t$ and given by $X(\alpha)=t$ and $X(\beta)=1$. Note that this grading is connected. Let $Y$ be the quotient grading of $X$ with  structural group the cyclic group of order $2$ and let $\mu:\Gamma(X)\to \Gamma(Y)$ be the quotient map of groups. Let $J$ be the automorphism of $\K$ which interchanges $\alpha$ and $\beta$. It is easy to verify that there exists a functor $H:  \K\#X\to\K\#Y$ such that $(H,J):F_X\to F_Y$ is a morphism of the Galois coverings; it provides $\mu$ as  morphism of gradings. Moreover  the morphism of coverings obtained with the identity automorphism of $\K$ also provides $\mu$ as a morphism of gradings.

Another example is provided by the set of automorphisms  $J_{p,q}$ of $X$ where $p$ and $q$ are non-zero scalars,    $J_{p,q}(\alpha)=p\alpha$ and $J_{p,q}(\beta)=q\beta$. All the  corresponding automorphisms $\left(H_{p,q}, J_{p,q}\right)$ of $F_X$ provide the identity as morphism of gradings.
\end{rem}

\section{Schurian generated categories}\label{sgc}

Let $\B$ be a $k$-category.  By definition, a non-zero morphism from an object $b$ to an object $b'$ is called \textbf{Schurian} if the space of all morphisms in the category from $b$ to $b'$ is one-dimensional.  The \textbf{Schurian generated subcategory} of $\B$ is
the intersection of all the $k$-subcategories of $\B$ containing the Schurian morphisms.
Its morphisms are  the sums of compositions of Schurian morphisms.
A $k$-category is called \textbf{Schurian generated} (\textbf{SG} for short) if it coincides with its Schurian generated subcategory.

Our main purpose in this section is to prove that a connected SG-category $\B$ admits a universal grading; actually  first we will prove that $\B$ admits  a connected grading by its  fundamental group.

Recall that a \textbf{Schurian category} is a $k$-category such that all the morphism spaces are one-dimensional - see \cite{bu1,CRS DOC 11}; a Schurian category is clearly Schurian generated. From the cited papers we know that a Schurian category admits a universal covering.

Other examples of SG-categories are provided by constricted categories $\B$ as follows. Recall that a presented $k$-category by a quiver $Q$ with  relations is a category of the form $kQ/I$  where $kQ$ is the linearization of the free category determined by $Q$ and $I$ is a two-sided  ideal contained in the square of the two-sided ideal generated by the arrows. Note that isomorphic categories may admit different presentations. It is easy to see that the arrows of the quiver generate the presented category. A presented category $kQ/I$ is called \textbf{constricted}  --  see \cite{bama} -- in case for each arrow $a$ any strictly parallel path (i.e. any path in $Q$ different from $a$ but sharing the same source and target with $a$) is zero in the quotient. Clearly this insures that arrows of $Q$ are Schurian morphisms, hence the presented category is Schurian generated.

\begin{defi}
Let $\B$ be an SG-category.
A \textbf{Schurian generated morphism} (\textbf{SG-morphism} for short) of $\B$ is a morphism which is a non-zero composition of Schurian morphisms; a \textbf{virtual SG-morphism} is a virtual morphism  where the morphism involved is an SG-morphism. An \textbf{SG-walk} is a walk made of  virtual SG-morphisms.
\end{defi}

\begin{lem} \label{SG homogeneous}
Let $\B$ be an SG-category.
\begin{enumerate}
\item
SG-walks are homogeneous with respect to any grading.
\item
Let $J$ be an automorphism which is the identity on objects,  $f$ an SG-morphism and  $X$ a grading. Then $J(f)$ is an SG-morphism and $\deg_X J(f) = \deg_X f$.
\end{enumerate}
\end{lem}

\begin{proof}
Let $X$ be a grading. A Schurian morphism from an object $b$ to an object $b'$ is clearly homogeneous since $_{b'}\B_{b}$ is one-dimensional.
An SG-morphism is a composition of Schurian morphisms, then it is homogeneous of degree the product of the degrees.

Moreover if $J$ is an automorphism which is the identity on objects and $f$ is a Schurian morphism, then $J(f)$ is homogeneous of the same degree. This also holds for SG-morphisms.
\qed
\end{proof}

\begin{rem}
Observe that assigning a group element to each one-dimensional space of Schurian morphisms does not  always produce a grading of the entire category. Indeed a given morphism may  be written  in several ways as a sum of compositions of SG-morphisms.
\end{rem}

\begin{lem}\label{SG connexe}
Let $\B$ be a connected SG-category. Between any two objects there is at least one SG-walk.
\end{lem}
\begin{proof}
Since $\B$ is connected between any two objects there exists a walk $w$. The first component of each virtual morphism of $w$ is a sum of SG-morphisms. Replacing this sum by one of its summands and performing this  for each virtual morphism provides a new walk $w'$ which is an SG-walk. \qed
\end{proof}

\begin{lem}\label{homogeneous morphisms in a SG}
Let $\B$ be an SG-category with a grading $X$. Each $X$-homogeneous morphism is a sum of SG-morphisms of same $X$-degree.
\end{lem}
\begin{proof}
Let $f$ be a non-zero morphism. Since $\B$ is Schurian generated  $f=\sum f_i$ where each $f_i$ is an SG-morphism, hence  homogeneous for any grading.  We write  $f=\sum g_j$ where each $g_j$ is non-zero and is the sum of all the $f_i$'s having  same $X$-degree. The $g_i$'s are $X$-homogeneous of different degrees, which means that they are the $X$-homogeneous components of $f$. Assume now $f$ is $X$-homogeneous. Then  the sum of the $g_j$'s is reduced to one summand, that is  $f$ is a sum of SG-morphisms of same $X$-degree.
\qed
\end{proof}

\begin{pro}\label{constant}
Let $\B$ be an SG-category with connected gradings $X$ and $Y.$ If there exists an automorphism  $J$  of $\B$
homogeneous from $X$ to $Y,$ then the identity of $\B$ is also homogeneous from $X$ to $Y$ and
\[\deg_Y \circ HW(J)= \deg_Y \circ HW(\mathds{1}). \]

\end{pro}

\begin{proof}
It suffices to prove the equality by evaluating  in an $X$-homogeneous morphism $f$. By the previous lemma
$f=\sum f_i$ where the $f_i$'s are SG-morphisms of same $X$-degree. We can assume that this expression is of minimal
length, in particular no subsum is zero. Now $J(f) = \sum J(f_i) = \sum \lambda_i f_i$, where the $\lambda_i$'s are
non-zero elements of $k$.
Note that in this expression of $J(f)$, no subsum is zero since $J$ is an automorphism.
Recall that $J(f)$ is $Y$-homogeneous.  Moreover  $\lambda_i f_i$ is an SG-morphism for all $i$ so it is $Y$-homogeneous as well.
Hence all the $f_i$'s have same $Y$-degree and $\deg_Y J(f)=\deg_Y f$.\qed
\end{proof}

\begin{cor}\label{unique}
Let $\B$ be an SG-category with connected gradings $X$ and $Y.$ There is at most one morphism of gradings from $X$ to $Y.$
\end{cor}
\begin{proof}
Let $\mu$ and $\mu'$ be morphisms corresponding respectively to automorphisms $J$ and $J'$ which are homogeneous from
$X$ to $Y$. By the previous result
\[\deg_Y \circ HW(J)= \deg_Y \circ HW(\mathds{1}) = \deg_Y \circ HW(J'), \]
hence $\mu \  \deg_X=\mu' \ \deg_X$. Since $\deg_X$ is surjective we obtain $\mu=\mu'$.\qed

\end{proof}

Let $\B$ be an SG-category with base object  $b_0$. Using Lemma \ref{SG connexe} we choose a family $v=(v_b)_{b\in\B_0}$ of SG-walks (called \textbf{connectors}) where $v_b$ goes from $b_0$ to $b$ and with the special choice of ${v}_{b_0}$ being the identity endomorphism at $b_0$. As in the proof of Proposition \ref{same mu}, given a morphism $f$ we set
 $${}^v\!f ={v}_{\tau(f)}^{-1} f {v}_{\sigma(f)}$$
 which is a closed walk at $b_0$. Observe that if $f$ is an SG-morphism then ${}^v\!f$ is a closed SG-walk.

\begin{pro}
Let $\B$ be a connected SG-category with base object $b_0$. There is a grading $P$ of $\B$ with structural group $\pigr$.
\end{pro}

\begin{proof} In order to define a grading $P$ by the fundamental grading group we first define the degree of an SG-morphism $f$:
\begin{equation}
\label{P}\tag{3}
\deg_P f = \left(\deg_X {}^v\!f\right)_X = \left(\deg_X v_{\tau(f)}^{-1} \ \deg_Xf\  \deg_Xv_{\sigma(f)}\right)_X
\end{equation}
where $X$ runs over all connected gradings of $\B$.  Note that the morphisms involved are SG-morphisms hence they are homogeneous for any connected grading by the first part of Lemma \ref{SG homogeneous}. Consequently the right hand side of the equality makes sense. We check now that the above family is graded coherent.

Let $\mu:X\to Y$ be a morphism of gradings. By definition there exists a homogeneous automorphism $J$ such that
$$\mu(\deg_X {}^v\!f)=\deg_Y \left(HW(J)({}^v\!f)\right).$$
Proposition \ref{constant} insures that
 $$\deg_Y \left( HW(J)({}^v\!f)\right)=\deg_Y {}^v\!f.$$

Secondly for any $p\in\pigr$, the homogeneous component of degree $p$  is defined as the set of all  the sums of SG-morphisms of degree $p$.  Since any morphism is a sum of SG-morphisms it only remains to prove that the sum of the subspaces is direct. Assume $f_1+\dots +f_n =0$ where the  $f_i$'s are morphisms of distinct $P$-degrees, our purpose is to prove $f_1=f_2=\dots=f_n=0$. In case $n>1$ and since $\deg_P f_1\neq\deg_P f_2$ there exists a connected grading $X_0$  such that $\deg_{X_0} {}^v\!f_1\neq \deg_{X_0} {}^v\!f_2$. Let $$I_1 =\{i\ \mid\ \deg_{X_0}{}^v\!f_i = \deg_{X_0}{}^v\!f_1\} $$
and let $I_2$ be its complement in $\{1,\dots,n\}$. Note that $1\in I_1$ and $2\in I_2$. Since $X_0$ is a grading we infer $\sum_{i\in I_1} f_i =0$ and  $\sum_{j\in I_2} f_j =0$. The result follows by induction.
\qed
\end{proof}

\begin{lem}\label{identity is homogeneous} The identity automorphism is homogeneous from $P$ to any connected grading $X$. Moreover for any $P$-homogeneous morphism $f$
\[ \deg_P f = \left(\deg_X {}^v\!f\right)_X.\]
\end{lem}
\begin{proof}
Let $f$ be a $P$-homogeneous morphism. By Lemma \ref{homogeneous morphisms in a SG} we know that  $f=\sum_if_i$ where  the $f_i$'s are  SG-morphisms of the  same $P$-degree. As a consequence for any connected grading $X$ and for any pair of indices $i$ and $j$
$$ \deg_{X} {}^v\!f_i= \deg_{X} {}^v\!f_j.$$
Consequently  ${}^v\!f= \sum {}^v\!f_i$ is $X$-homogeneous for any connected grading $X$ and $\deg_X{}^v\!f=\deg_X{}^v\!f_i$ for any $i$. Finally
$$\deg_P f= \deg_P f_i=\left(\deg_X {}^v\!f_i\right)_X = \left(\deg_X {}^v\!f\right)_X.$$  \qed
\end{proof}

Let $w$ be a $P$-homogeneous walk and let ${}^v\!w$ be the closed walk defined as before for morphisms.

\begin{lem}\label{connectors trivial degree}
A $P$-homogeneous walk $w$ is $X$-homogeneous for every connected grading $X$. Its $P$-degree is given by the same formula as above, namely:
$$\deg_P w = \left( \deg_X {}^v\!w\right)_X.$$
In particular the connectors are $P$-homogeneous of trivial $P$-degree.
\end{lem}
\begin{proof}
 The $P$-degree of $w$ is the product of the $P$-degrees of the homogeneous virtual morphisms involved which can be computed according to Lemma \ref{identity is homogeneous} and  equality \eqref{P}. Notice that the $X$-degrees of the connectors annihilate themselves.

 Finally the connectors are SG-walks, hence they are $P$-ho\-mo\-ge\-neous walks and their degree can be computed using the formula just proved and the fact that the connector at $b_0$ is the identity which is of trivial degree for any grading. \qed
\end{proof}

\begin{thm}\label{SG has universal grading}
A connected SG-category $\B$ with base object $b_0$ admits a universal grading with structural group isomorphic to $\Pi_1(\B,b_0)$.
\end{thm}

\begin{proof}
Recall from Remark \ref{image is a group} that $\Im\ \deg_P$ is a subgroup of the structural group. We claim that  $P$  can be restricted to a grading with structural group  $\Im\ \deg_P$ which we will denote $P\!\!\downarrow$. To this end it suffices  to check that the $P$-degree of a Schurian morphism $f$ is the $P$-degree of some closed $P$-homogeneous walk at $b_0$.   Note that since $P$ is a grading we have the following:
$$\deg_P {}^v\! f= \deg_Pv_{\tau(f)}^{-1}\ \deg_Pf\  \deg_Pv_{\sigma(f)}.$$
Since by Lemma \ref{connectors trivial degree}  the $P$-degrees of the connectors are trivial,  $\deg_Pf=\deg_P{}^v\!f.$
Consequently the grading $P\!\!\downarrow$ with structural group $\Im\ \deg_P$ exists. By construction this grading is connected. Next we will prove it is universal.

Let $X_0$ be a fixed connected grading of $\B$. Our purpose is to show the existence of a unique morphism of gradings $\mu:P\!\!\downarrow\to X_0$, that is the existence of a unique group morphism
$ \mu:\Gamma\left(P\!\!\downarrow\right)\to \Gamma(X_0)$ such that there exists at least one homogeneous automorphism $J$ of $\B$ making  the following diagram commutative:
$$
\xymatrix{
{}_{b_0}HW(\B,P\!\!\downarrow)_{b_0} \ar@{->>}[d]_{\deg_{P\!\downarrow}} \ar[r]^{ HW(J)} & {}_{b_0} HW(\B,X_0)_{b_0}\ar@{->>}[d]^{ \deg_{X_0}}\\
\Gamma(P\!\!\downarrow)=\Im\ \deg_P \ar[r]_{\mu}   &\Gamma({X_0}).}
$$
Consider the  group morphism $\mu$ defined as follows: let $\gamma$ be an element in $\Im\ \deg_P$, that is there exists   a closed $P$-homogeneous walk $w$ at $b_0$ such that $\deg_Pw=\gamma$. By Lemma \ref{connectors trivial degree}
$$\deg_P w= \left( \deg_X {}^v\!w\right)_X$$
and the latter equals $\left( \deg_Xw\right)_X$ since $w$ is already a closed walk at $b_0$. We set
$$\mu(\gamma)= \deg_{X_0}w.$$
Note that $\mu$ is well defined: let $w'$ be another closed $P$-homogeneous walk at $b_0$ representing $\gamma$ that is $\deg_Pw'=\gamma=\deg_Pw$, then $\deg_{X_0}w'=\deg_{X_0}w$.

According to Lemma \ref{identity is homogeneous} the identity automorphism of $\B$ is homogeneous from $P\!\!\downarrow$ to $X_0$. The morphism $\mu$ above makes the diagram with $J=\mathds{1}$ commutative.

 The above morphism is unique since by Corollary \ref{unique} between two connected gradings of an SG-category there is at most one morphism.

Theorem \ref{universal fundamental} asserts that the structural group of the universal grading is isomorphic to the fundamental group of the category.

 \qed
\end{proof}

\section{Examples}

We consider four examples. As mentioned in the Introduction the first one is at the source of the theory of the fundamental group \emph{\`{a} la Grothendieck}, it has a universal grading with cyclic fundamental group.  The second example has no universal grading but admits a versal grading and the fundamental group is trivial. Then we show that for a presented monomial Schurian $k$-category the fundamental  group of the presentation is isomorphic to the intrinsic fundamental group. Finally we consider an example in characteristic $p$  which has neither universal nor versal grading; its intrinsic fundamental group is the product of the infinite cyclic group and  the cyclic group of order $p$.

\subsection{A one-parameter family}\label{B_q}

Let $Q$ be the quiver

$$
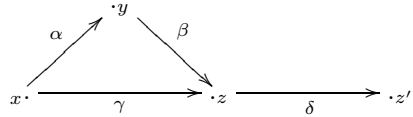
\xymatrix{
&._y\ar[rd]^\beta\\
{}_{x}.\ar[rr]_\gamma\ar[ru]^\alpha&&._z\ar[rr]_\delta&& ._{z'} } $$
and let $kQ$ be the linearization of its path category.
For each  $q\in k$ the one-dimensional vector space  $I_q = k(\delta\gamma-q\delta\beta\alpha)$ is a two-sided ideal of $kQ$. We denote $\B_q$ the quotient $k$-category $kQ/I_q$. Note that $\gamma$ is not a Schurian morphism, and does not belong to the Schurian generated subcategory.

We first recall that $\B_q$ is isomorphic to $\B_{q'}$ for all $q$ and $q'$. Indeed, let $F$ be the automorphism of $kQ$ which is the identity on the objects,  $$F(\gamma) = \gamma+ (q-q')\beta\alpha$$
and is the identity on the other arrows. Clearly  $F(I_q)=I_{q'}$ and it induces an isomorphism $\hat{F}:\B_q\to\B_{q'}$. In particular $\B_q$ is isomorphic to $\B_0$ for any $q$, in other words $\B_q$ is a trivial one-parameter deformation of $\B_0$.

\begin{rem}
 The non-intrinsic fundamental group considered by Bongartz, Gabriel \cite{boga} and Martinez-Villa, de la Pe\~{n}a \cite{MP} for representation theory purposes relies on a particular presentation by a quiver with relations of the $k$-category. In case $q\neq 0$ this non-intrinsic fundamental group is trivial  while it is the infinite cyclic group for $q=0$. One of the purposes of the fundamental group \emph{\`{a} la Grothendieck} we have considered is to provide a theory which does not depend on a presentation.
\end{rem}

 We will perform computations for an arbitrary value of the parameter $q$ although we know that it suffices to do it for a specific value, for instance $q=0$. We  do so in order to confirm in this  example the theory of the fundamental grading group in its intrinsic aspect, namely that it only depends on the isomorphism class of the $k$-category.

\begin{pro}
For any  non-trivial connected grading $X$ of $\B_q$ the morphism $\gamma - q\beta\alpha$ is homogeneous and the structural group is cyclic.
\end{pro}
\begin{proof}
Let $X$ be a connected grading.
Since $\alpha$, $\beta$ and $\delta$ are Schurian morphisms they are homogeneous of degrees denoted respectively $a$, $b$ and $d$. The space of morphisms from $x$ to $z$ is two-dimensional and $\beta\alpha$ is already homogeneous, let $\gamma + \l\beta\alpha$ be the other homogeneous morphism complementing $\beta\alpha$; we denote by $c$ its degree.

If $c=ba$ the homogeneous closed walks at $x$ are all of trivial degree, implying that the structural group is trivial since $X$ is connected.

 Note that $\delta(\gamma+l\beta\alpha)$  is homogeneous of degree $dc$. Moreover
$$\delta(\gamma+l\beta\alpha)= (q+l)\delta\beta\alpha.$$
Consequently if  $q+l\neq 0$ the degrees of those morphisms coincide, namely  $dc=dba$, then $c=ba$ and the grading is trivial.
If $X$ is not trivial then $l=-q$, the morphism $\gamma - q\beta\alpha$ is homogeneous of degree $c$ and $c\neq ba$. Consider
$$\deg_X: {}_xHW(\B_q,X)_x\to\Gamma (X).$$
Then $\Im\ \deg_X=\{ \left(c^{-1}ba\right)^i\  \mid\ i\in\ \mathds{Z}\}$. Since $X$ is connected $\Im\ \deg_X=\Gamma(X)$.\qed
\end{proof}

 Let $U$ be  the grading of $kQ$ by the infinite cyclic group $T=<t>$  such that  $\alpha$, $\beta$ and $\delta$  are of  trivial degree while $\gamma - q\beta\alpha$ is of degree $t$. The ideal $I_q$ is homogeneous so the grading is well defined on $\B_q$ and we still denote it by $U$.

\begin{pro}
The grading $U$ of $\B_q$ is universal.
\end{pro}

\begin{proof}
For a given non-trivial connected grading $X$ we will prove the existence of  a unique morphism of gradings $\mu : U\to X$. Note that the identity automorphism is homogeneous from $U$ to $X$. We use the same notations as in the previous proof. The group map $\mu : T\to \Gamma(X)$ is given by $\mu(t)=c^{-1}ba$.  The following diagram is commutative:
$$
\xymatrix{
{}_{x}HW(\B,U)_{x} \ar@{->>}[d]_{\deg_U} \ar[r]^{HW(\mathds{1})} & {}_{x}HW(\B,X)_{x}\ar@{->>}[d]^{\deg_X}\\
T \ar[r]_{\mu}   &\Gamma(X).}
$$
In order to prove that the morphism is unique let $X$ and $Y$ be non-trivial connected gradings. As before,  the identity automorphism is homogeneous from $X$ to $Y$. We assert that for any homogeneous automorphism $J$ of $\B_q$ from $X$ to $Y$ the following holds:
\begin{equation}
\label{as identity}\tag{4}
\deg_Y \circ HW(J)  = \deg_Y \circ HW(\mathds{1}).
\end{equation}

To this end, we show that $J$  is multiplication by non-zero scalars when evaluated on homogeneous morphisms. Indeed the homogeneous components are one-dimensional and coincide for $X$ and $Y$. Then $J$ multiplies by non-zero scalars the Schurian morphisms $\alpha$,  $\beta$ and $\delta$ as well as their possible compositions. Since $J$ is homogeneous  $J(\gamma - q\beta\alpha)$ is homogeneous but cannot be a scalar multiple of $\beta\alpha$ because this morphism is already in the image of $J$.  Then $J(\gamma - q\beta\alpha)$ is a non-zero scalar multiple of $\gamma - q\beta\alpha$.

Finally let $\mu': U\to X$ be a morphism of gradings with corresponding homogeneous automorphism $J$; we have just proved  that $\deg_X\circ HW(J)=\deg_X\circ HW(\mathds{1})$, then $\mu'\ \deg_U = \mu\ \deg_U$ and $\mu=\mu'$ since $\deg_U$ is surjective.\qed
\end{proof}
\begin{cor}
For any  $q\in k$ the intrinsic fundamental group $\Pi_1(\B_q,x)$ is infinite cyclic .
\end{cor}

\subsection{The Kronecker category}

We will prove that the Kronecker category $\K$ considered in Remark \ref{K} does not have a universal grading, instead there exists a unique versal grading with structural group the infinite cyclic group. Its fixed subgroup and the intrinsic fundamental group of $\K$ are trivial.

\begin{lem}
Every non-trivial connected grading $X$ of $\K$ has cyclic structural group and it is determined  by the choice of two linearly independent morphisms $\alpha$,$\beta$  in ${}_y\K_x$ and the assignment of their  degrees $a$ and $b$ verifying that $b^{-1}a$ is a generator of $\Gamma(X)$.
\end{lem}
\begin{proof}
Let $X$ be a connected grading of $\K$. In case the entire ${}_y\K_x$ is homogeneous every closed walk at $x$ is homogeneous of trivial degree, hence the connected grading is trivial since the degree map is surjective.

In case $X$ is not trivial, there exist two one-dimensional homogeneous components and we can choose a homogeneous basis $\{\alpha, \beta\}$ with  distinct degrees denoted $a$ and $b$. Observe that the degrees of the closed homogeneous walks at $x$ are powers of $b^{-1}a$, which shows that $\Gamma(X)$ is cyclic generated by $b^{-1}a$.

Conversely a non-trivial connected grading can be constructed following this pattern once a basis is given as well as two elements in the cyclic group such that their difference is a generator of it.\qed
\end{proof}

\begin{pro}
$\K$ does not admit a universal grading.
\end{pro}
\begin{proof}
Let $X$ be a non-trivial connected grading. We assert that $X$ has an automorphism of gradings with group map $\mu$ determined by $$\mu(b^{-1}a)=\left(b^{-1}a\right)^{-1}=a^{-1}b$$ and homogenous automorphism $J$ given by $$J(\alpha)=\beta\ \mbox{and}\ J(\beta)=\alpha$$
making  the required diagram commutative. This automorphism is not the identity unless the structural group is  of order $2$, hence $X$ is not universal out of this case.

If the structural group is finite cyclic (in particular of order $2$) the grading is not universal since there is no map from it to any grading by the infinite cyclic group.\qed

\end{proof}

Let $T$ be an infinite cyclic group with a given generator $t$. We consider a grading $V$ of $\K$ with structural group $T$ as follows. Let  $\{\alpha,\beta\}$ be a basis of ${}_y\K_x$ and we set $\deg_V\alpha =t$ and $\deg_V\beta =1$. Note that $V$ is connected.

\begin{pro}
The grading $V$ is versal. Moreover $\mathsf{Fix}(V)$ and $\Pi_1(\K,x)$ are trivial.
\end{pro}
\begin{proof}
Let $X$ be a non-trivial connected grading with homogeneous linearly independent morphisms $\alpha'$ and $\beta'$ of  degrees $a$ and $b$. The group map $\mu: T\to \Gamma(X)$ given by $\mu(t)=b^{-1}a$ is a morphism of gradings. Indeed  the homogeneous automorphism $J$ determined by $J(\alpha)=\alpha'$ and $J(\beta)=\beta'$ makes  the required diagram commutative.

In the proof of the preceding proposition we have shown that $V$ admits a non-trivial grading automorphism with group map sending $t$ to $t^{-1}$ hence its structural group does not have fixed elements except $1$. By Proposition \ref{fix} the fundamental group is also trivial. \qed
\end{proof}

Note that all the connected gradings with infinite cyclic group are isomorphic by  analogous computations as  before. Hence the versal grading is unique up to isomorphisms of gradings.

\subsection{Monomial Schurian categories}
By definition a monomial presentation of a $k$-category is as follows: let $Q$ be a finite quiver, $kQ$ be the path category  and $I$ be a two-sided ideal of $kQ$ generated by a set of paths of length at least $2$. Assume moreover that all paths of a given length $n$ belong to $I$.

Then the  fundamental group of the presentation, as considered in \cite{boga,MP},  is clearly isomorphic to the topological fundamental group of the graph underlying $Q$.

Recall from \cite{CRS DOC 11} that  a $k$-category $\B$ is Schurian if ${}_y\B_x$ is zero ore one-dimensional for any objects $x$ and $y$. In that case we have proved in this paper that the intrinsic fundamental group of $\B$ is isomorphic to the topological fundamental group of the  associated  CW-complex (Definition 3.1).

Assume now that $B=kQ/I$ is monomial and Schurian. Then it is easy to prove that the CW-complex and the underlying graph of $Q$ have the same homotopy type. Consequently the fundamental group of the presentation and the intrinsic fundamental group are isomorphic.

\subsection{Truncated polynomial algebras in finite characteristic}

Let $k$ be a field of characteristic $p$, let $C_p=<t\ \mid\ t^p=1>$ be the cyclic group of order $p$ and let $B=kC_p$ be the group algebra which is isomorphic to $k[x]/(x^p)$ through the isomorphism assigning $(x-1)$ to $t$. We consider $B$ as a $k$-category with a single object having $B$ as endomorphisms.

Connected gradings of $B$ have been studied in detail in \cite[Section 5, p. 640]{CRS ANT 10}. There are two families of non-trivial connected gradings as follows:

\begin{itemize}
\item
\textbf{Invertible.} In case there exists a non-scalar invertible $X$-homogeneous element in $B$, then all the $X$-homogeneous elements are invertible and the structural group is $C_p$. So $X$ is  isomorphic to the natural grading of $kC_p$. Moreover $X$ is simply connected, which means that if $Y$ is a connected grading and $\mu:Y\to X$ is a morphism then $\mu$ is an isomorphism.
\item
\textbf{Maximal.} In case every $X$-homogeneous non scalar element belongs to the maximal ideal $(x)$, the structural group is cyclic and $X$ is isomorphic to a quotient of the natural grading of $k[x]/(x^p)$ where $x$ is homogeneous. When the structural group is infinite cyclic the grading is simply connected.
\end{itemize}

Any automorphism of $B$ preserves the maximal ideal, hence there are no morphisms of gradings between both families. There is no universal nor versal grading and  the fundamental group is $T\times C_p$.


\footnotesize
\noindent C.C.:
\\Institut de math\'{e}matiques et de mod\'{e}lisation de Montpellier I3M,\\
UMR 5149\\
Universit\'{e}  Montpellier 2,
\\F-34095 Montpellier cedex 5,
France.\\
{\tt Claude.Cibils@math.univ-montp2.fr}

\noindent M.J.R.:
\\Departamento de Matem\'atica,
Universidad Nacional del Sur,\\Av. Alem 1253\\8000 Bah\'\i a Blanca,
Argentina.\\ {\tt mredondo@criba.edu.ar}

\noindent A.S.:
\\IMAS y Departamento de Matem\'atica,
 Facultad de Ciencias Exactas y Naturales,\\
 Universidad de Buenos Aires,
\\Ciudad Universitaria, Pabell\'on 1\\
1428, Buenos Aires, Argentina. \\{\tt asolotar@dm.uba.ar}

\end{document}